[1994/12/01]
\documentclass[11pt]{article}
\title{Analytically weak solutions to SPDEs with unbounded time-dependent differential operators and an application}
\author{\textsc{B. Baur\textsuperscript{A}, M. Grothaus\textsuperscript{A} and T. T. Mai\textsuperscript{A,B} }\footnote{Benedict Baur (baur@mathematik.uni-kl.de), Martin Grothaus (grothaus@mathematik.uni-kl.de), Thanh Tan Mai (tan@mathematik.uni-kl.de)}
\\[2.5ex]
\textsuperscript{A} \footnotesize{University of Kaiserslautern, Department of Mathematics, P.O.~Box 3049, 67653 Kaiserslautern, Germany}\\
\textsuperscript{B} \footnotesize{University of Quynhon, Department of Mathematics, 170 An Duong Vuong, Quy Nhon, Vietnam}
}
\date{\today}

\usepackage{amsmath,amsthm,amssymb,amsfonts,url}
\usepackage[mathscr]{eucal}
\usepackage{geometry}

\chardef\bslash=`\\ 

\newtheorem{theorem}{Theorem}[section]

\newtheorem{lemma}[theorem]{Lemma}
\newtheorem{proposition}[theorem]{Proposition}

\theoremstyle{definition}
\newtheorem{definition}[theorem]{Definition}
\newtheorem{assumption}[theorem]{Assumption}
\newtheorem{remark}[theorem]{Remark}

\newcommand{\R}{\mathbb{R}}
\newcommand{\N}{\mathbb{N}}

\renewcommand{\P}{\mathbb{P}}
\newcommand{\E}{\mathbb{E}}

\newcommand{\F}{\mathcal{F}}

\newcommand{\D}{\mathcal{D}}

\newcommand{\e}{{\bf{e}}}
\newcommand{\z}{{\bf{z}}}
\newcommand{\f}{{\bf{f}}}
\newcommand{\0}{{\bf{0}}}
\newcommand{\w}{{\bf{w}}}
\newcommand{\x}{{\bf{x}}}
\newcommand{\y}{{\bf{y}}}

\renewcommand{\z}{{\bf{z}}}
\renewcommand{\u}{{\bf{u}}}
\renewcommand{\v}{{\bf{v}}}

\newcommand{\eval}[2][\right]{\relax
  \ifx#1\right\relax \left.\fi#2#1\rvert}

\begin{document}
\maketitle

\begin{abstract}
We analyze the concepts of analytically weak solutions of stochastic differential equations (SDEs) in Hilbert spaces with time-dependent unbounded operators
and give conditions for existence and uni\-queness of such solutions. Our studies are motivated by a stochastic partial differential
equation (SPDE) arising in industrial mathematics.
\end{abstract}


\noindent {\bf AMS classification (2000):} Primary: 65J08, 60H15; Secondary: 60H05, 58D25

\section{Introduction}

Let $G,H$ be separable Hilbert spaces and $W=(W(t))_{0 \le t \le T}$, $0<T<\infty$, be a $G$-valued $Q$-Wiener process, see e.g.~\cite{DaPr}, on a
filtered probability space $(\Omega, \F, (\F_t)_{0 \le t \le T}, \P)$. We consider the equation 
  \begin{equation}\label{1}
    \begin{split}
       dX(t)&=\big(L(t)X(t)+F(t)\big)dt+AdW(t),\quad 0\leq t_0\leq t\leq T,\\
       X(t_0)&=\xi,
    \end{split}
 \end{equation}    
where $L(t): D(L(t))\subset H\rightarrow H$, $t\in[t_0,T]$, are closed linear operators, densely defined on $H$, $A\in L(G,H)$
(space of linear continuous mappings from $G$ to $H$),
$F=(\F_t)_{t_0\leq t\leq T}$ an $H$-valued process, pathwise Bochner integrable on $[0,T]$, and $\xi$ is an $\F_{t_0}$-measurable $H$-valued random variable.	

There are several textbooks and articles on the type of equations as in (\ref{1}). Da Prato and Zabczyk in \cite{DaPr} considered the case
$L(t)=L(t_0)$, $t\in[t_0,T]$, i.e.~the operators are constant in time. Manthey and Zausinger in \cite{MZ} provided mild solutions to
(\ref{1}) for the case $H$ being a weighted $L^p$ space. Pr{\'e}v{\^o}t and R{\"o}ckner for coercive $L(t)$ constructed variational solutions
to (\ref{1}), see \cite{PR}. Veraar and Zimmerschied in \cite{VZ} considered the case where the $L(t)$ are sectorial, uniformly in $t \in [t_0,T]$.

Our studies are motivated by a stochastic partial differential equation arising in industrial mathematics. When reformulated as in
(\ref{1}), the corresponding $(L(t), D(L(t)))$, $t \in [t_0,T]$, form a family of unbounded operators
on an appropriate Sobolev space $H$. Hence, the results as in \cite{MZ} are not applicable. Furthermore, the
operators $(L(t), D(L(t)))$, $t \in [t_0, T]$, are neither coercive nor sectorial. Hence, we can not use the results from \cite{PR} or \cite{VZ}.

Instead, we generalize the concepts and results of \cite{DaPr} to the case of time-dependent operators. More precisely, we generalize the notion of an \emph{analytically weak solution} to the time-dependent case and prove existence and uniqueness of such solutions under certain assumptions. In particular, every weak solutions is given by the mild solution.

This article is organized as follows. In Section \ref{sec2} we consider analytically weak solutions to \eqref{1}.
Our existence result we proof in Subsection \ref{subsec21}, see Theorem \ref{existence} below, and the uniqueness result, see Theorem \ref{main}
below, in Subsection \ref{subsec22}. The assumptions we impose allow time-dependent unbounded operators on separable Hilbert spaces which can be
non-coercive and might be non-sectorial. Of course, our results are based on the concepts of non-time-homogeneous evolution systems.    

Finally, in Section \ref{apply}, we apply our results to a linear stochastic partial differential equation (SPDE) arising in industrial mathematics. In \cite{MW1}, see also \cite{Mar} and \cite{MW} for a derivation of the deterministic equation,
the following equation for modeling the behavior of a fiber under influence of a turbulent air-flow is derived:
\begin{multline}\label{pdae}
d \, \partial_t \x(s,t) = \big(\partial_s(\lambda\partial_s\x)(s,t) - b\partial_{ssss}\x(s,t) \\
-g\e_3+\f^{{\rm det}}(s,t)\big)dt+\sigma d\w(s,t), \quad (s,t)\in[0,l]\times [0,T],
\end{multline}
with initial condition
\begin{equation}\label{inicond}
\x(s,0) = (s-l)\e_3,\quad \partial_t\x(s,0)=\0,\quad s \in[0,l], \tag{\ref{pdae}a}
\end{equation}
boundary condition
  \begin{equation}\label{boundcond}
    \x(l,t)=\0,\quad
    \partial_s\x(l,t)=\e_3,\quad
    \partial_{ss}\x(0,t)=\0,\quad
     \partial_{sss}\x(0,t)=\0,\quad t\in[0,T].\tag{\ref{pdae}b}
  \end{equation}
Here $(\w(t))_{0\leq t\leq T}$ is a $Q$-Wiener process on a filtered probability space $(\Omega,\F, (\F_t)_{0\le t \le T}, \P)$ and $\x(\omega): [0,l]\times [0,T]\longrightarrow \R^3$,
$\omega\in\Omega$, models the fiber at arc length $s\in[0,l]$ and time $t\in[0,T]$. The function
$\lambda: [0,l]\times [0,T] \longrightarrow [0, \infty)$ is the tractive force with the boundary condition $\lambda(0,t)=0$, $t\in[0,T]$,
and $\e_3=(0,0,1)$. $\f^{{\rm det}}:[0,l]\times [0,T]\rightarrow \R^3$ is a deterministic force, $0<b, g, \sigma <\infty$ are constants
(bending stiffness, constant of gravitation, amplitude of stochastic force). Equations of the type as in \eqref{pdae} in literature are
also called beam equations. 
In the mentioned articles, the equation is considered with the additional non-linear algebraic constraint
 \begin{equation}\label{algcond}\|\partial_s\x(s,t)\|_{{\rm euk}}=1\quad\text{for all}\ (s,t)\in[0,l]\times [0,T]. \tag{\ref{pdae}c}
 \end{equation}
In this article we restrict ourselves to the linear equation, i.e., we do not consider the algebraic constraint. 

For complementary results on stochastic beam equations (with operators not being time-dependent and without algebraic constraint) we refer to
\cite{DM03}, \cite{DS05} and \cite{BMS05} and references therein.

In Subsection \ref{homogen} we consider \eqref{pdae} with general initial condition, but homogeneous boundary condition, i.e.,
conditions as in \eqref{boundcond}, but $\partial_s\x(l,t)=\0$, $t \in [0,T]$. Under Assumption \ref{ass}, see below, we can prove that it has
a unique analytically weak solution, see Theorem \ref{12101} below.
Technically, we have to show the existence a corresponding non-time-homogeneous evolutions system, having sufficient
properties in order to apply our concepts from Section \ref{sec2}. In Subsection \ref{nonhomogen} we provide a unique analytically weak solution
to \eqref{pdae}, \eqref{inicond}, \eqref{boundcond}, see Theorem \ref{12102} below.
The existence of tractive forces, such that the algebraic constrained \eqref{algcond}
is also fulfilled, is topic of ongoing research.


\section{Analytically weak solutions}\label{sec2}

In this article we fix $0 < T < \infty$. Furthermore, $(G,\langle\cdot,\cdot\rangle_G)$ and $(H,\langle\cdot,\cdot\rangle_H)$ are
separable Hilbert spaces with corresponding norms $\|\cdot\|_G:=\sqrt{\langle\cdot,\cdot\rangle_G}$ and $\|\cdot\|_H:=\sqrt{\langle\cdot,\cdot\rangle_H}$,
respectively. We denote by $(L(G,H),\|\cdot\|_{L(G,H)})$ the Banach space of bounded linear operators from $G$ to $H$, where $\|\cdot\|_{L(G,H)}$
is the operator norm. We use the notation $L(H)$ in the case $G = H$. If there is no danger of confusion, we drop the subindex. Assume that
$L: D(L)\subset H\longrightarrow H$ is a densely defined linear operator. Then $(L^*, D(L^*))$ denotes the adjoint of $(L,D(L))$
with respect to $\langle \cdot, \cdot \rangle_H$. The graph norm on $D(L)$ w.r.t.~the operator $(L,D(L))$ is denoted by
$\|\cdot\|_{D(L)}.$ When applying our results we use the concepts of stable family of operators, part of an operator in some subspace,
invariant and admissible subspaces as in \cite{Pazy}. The measurability of $L(G,H)$-valued functions is considered as in \cite{DaPr}. Partial derivatives in direction $x$, where $x$ is a real variable, are denoted by $\partial_x$. Right and left derivatives are denoted by $\partial^+_x$ and $\partial^-_x$, respectively. Higher order partial derivatives are denoted by $\partial_{xx}$, $\partial_{xxx}$ and so forth.
\begin{definition}\label{26101} Let $\big(L(t), D(L(t))\big)_{0 \le t \le T}$ be a family of densely defined closed linear (unbounded) operators
on $H$ such that $\D:=\cap_{0\leq t\leq T}D(L(t))$ is dense in $H$. A family $(U(t,\tau))_{0\leq\tau\leq t\leq T}$ of linear bounded operators on $H$ is called {\it almost strong evolution system} corresponding to the family $\big(L(t), D(L(t))\big)_{0 \le t \le T}$ with initial space
$Y\subset\D\subset H$, dense in $H$, if the following holds:

\noindent
(i) $U(t,t)= Id$ for all $t\in[0,T]$, and \ $U(t,r)U(r,\tau)=U(t,\tau)$\ for all $0\leq \tau\leq r\leq t\leq T.$
 
\noindent
(ii) $[\tau, T]\ni t\mapsto U(t,\tau)u\in H$,\ $[0, t]\ni \tau\mapsto U(t,\tau)u\in H$ are continuous for all $u\in H$ and all $0 \le \tau \le t \le T$ and $\sup_{0\leq\tau\leq t\leq T}\|U(t,\tau)\| <\infty$.

\noindent
(iii) $U(t,\tau)(Y)\subset D(L(t))$ for all $0 \le \tau <T $, a.e.~$t \in [\tau,T]$, and
  \begin{equation}\label{30101}
 \int_\tau^tL(r)U(r,\tau)udr=U(t,\tau)u-u\quad\text{for all}\ u\in Y,\ 0\leq\tau\leq t\leq T.
 \end{equation}
We call the family $\big(L(t), D(L(t))\big)_{0 \le t \le T}$ generator of the almost strong evolution system
$(U(t,\tau))_{0\leq\tau\leq t\leq T}$.
\end{definition}

\begin{remark} (i) If $(U(t,\tau))_{0\leq\tau\leq t\leq T}$ satisfies Definition \ref{26101} (iii), then $U(\cdot,\tau)u$ is differentiable
a.e.~on $[\tau,T]$ for all $u\in Y$ and $\partial_tU(t,\tau)u=L(t)U(t,\tau)u$ for a.e.~$t\in[\tau,T]$.

\noindent
(ii) Every evolution system as in \cite[Theo.~1]{KatII} or \cite[Theo.~5.4.3]{Pazy} is an almost strong evolution system.
\end{remark}

\subsection{Existence}\label{subsec21}

Let $(\Omega,\F, (\F_t)_{0\le t \le T}, \P)$ be a filtered probability space and $Q\in L(G)$
a nonnegative symmetric operator. In this section we assume $(W(t))_{0\leq t\leq T}$ to be a $Q$-Wiener process
on $(\Omega,\F, (\F_t)_{0\le t \le T}, \P)$, see e.g.~\cite{DaPr}, \cite{PR}. We use the notations $L_2(G,H)$ for the (separable Hilbert)
space of Hilbert-Schmidt operators and $L_2^0:=L_2(Q^\frac{1}{2}(G), H)$ the Cameron-Martin space associated to $Q$,
see e.g.~\cite{DaPr}, \cite{PR} for construction and details. We consider Equation (\ref{1}), i.e.,
  \begin{equation}\label{SDE}
    \begin{split}
       dX(t)&=\big(L(t)X(t)+F(t)\big)dt+AdW(t),\quad 0\leq t_0\leq t\leq T,\\
       X(t_0)&=\xi.
    \end{split}\tag{\ref{1}}
 \end{equation}
In this section the following is always assumed:
\begin{assumption}
\noindent
(i) $L(t): D(L(t))\subset H\longrightarrow H, t\in[t_0,T]$, is a family of densely defined closed linear operators.

\noindent
(ii) $\D:=\bigcap_{t\in[t_0,T]} D(L(t))$ and $\D^*:=\bigcap_{t\in[t_0,T]} D(L^*(t))$ are dense in $H$. 

\noindent
(iii) $A\in L(G,H)$ and $\xi$ is a $\F_{t_0}$-measurable $H$-valued random variable.	

\noindent
(iv) $F$ is an $H$-valued predictable process, pathwise Bochner integrable on $[t_0, T]$.
\end{assumption}

\begin{definition} An $H$-valued process $(X(t))_{t_0\leq t\leq T}$ is called an {\it analytically weak solution} of (\ref{SDE}) if  it is $H$-predictable, has $\P$-a.e.~(Bochner) square integrable trajectories and for all $h\in \D^*$, $t\in [t_0,T]$, we have
$$\langle X(t),h\rangle =\langle \xi,h\rangle +\int_{t_0}^t(\langle X(r),L^*(r)h\rangle +\langle F(r),h\rangle )dr+ \int_{t_0}^t\langle h, AdW(r)\rangle \quad \P{\rm -a.e.}.$$
\end{definition}
\begin{remark}(i) The stochastic integral $\int_{t_0}^t\langle h,AdW(r)\rangle$ is in the sense of \cite[Lemma 2.4.2]{PR}, i.e., $\int_{t_0}^t\langle h,AdW(r)\rangle = \int_{t_0}^t\langle h,A(\cdot) \rangle dW(r)$ is an $\R$-valued random variable. \label{RemStochasticIntegral}

\noindent
(ii) Concerning predictable Hilbert space valued random processes, see e.g.~\cite{DaPr}, \cite{PR}.
\end{remark}
\begin{theorem}\label{existence} (i) Let $(U(t,\tau))_{0\leq t_0\leq\tau\leq t\leq T}$ be an almost strong evolution system  on $H$ corresponding to the family of linear operators $\big(L(t), D(L(t))\big)_{t_0\leq t\leq T}$ and $Y$ in the sense of Definition \ref{26101} and let
\begin{equation}\label{07111}
\int_{t_0}^t\|U(t,r)A\|^2_{L_2^0}dr=\int_{t_0}^t{\rm Tr}\big(U(t,r)AQA^*U^*(t,r)\big)dr<+\infty,
\end{equation} 
where ${\rm Tr}(B)$ denotes the trace of a non-negative $B\in L(H)$.
Then the mild solution of (\ref{SDE}), defined by $$I(t,t_0):=U(t,t_0)\xi+\int_{t_0}^tU(t,r)F(r)dr+\int_{t_0}^tU(t,r)AdW(r),\ t\in[t_0,T],$$ exists.   

(ii) If we further assume that the map $[t_0,T]\ni t\mapsto L^*(t)h\in H$ is bounded and measurable for all $h\in \D^*$, the mild solution is also an analytically weak solution of (\ref{SDE}).
\end{theorem}
 
\begin{proof} W.l.o.g.~one just need to consider the case $\xi =0$, $F=0$, and $t_0=0$.

 (i): Due to (\ref{07111}) the stochastic integral $\int_{t_0}^tU(t,r)AdW(r),\ t\in[t_0,T]$, exists, 
see e.g. \cite[p.94]{DaPr}, \cite[p.~27,28]{PR}. Thus, (i) is shown. 

(ii): Let $t\in [0,T]$ and $h\in \D^*$, then by assumption there exists $0<C_1<\infty$ such that
      $$|\langle L^*(r)h,I(r,0)\rangle|\leq \| L^*(r)h\|\|I(r,0)\|\leq C_1\|I(r,0)\|\quad\text{for all}\ r\in[0,t].$$ 
      By (\ref{07111}), the function $[0,t]\ni r\longmapsto\|I(r,0)\|\in\R$ is integrable, see e.g.~\cite{DaPr}, \cite{PR}. So, the integral $\int_0^t\langle L^*(r)h,I(r,0)\rangle dr$ does exist. Moreover, for all $t\in[0,T]$, we have
     $$\int_0^t\langle L^*(r)h,I(r,0)\rangle dr=\int_0^t\langle L^*(r)h,\int_0^t1_{[0,r]}(r_1)U(r,r_1)AdW(r_1)\rangle dr,$$
     where $1_S$ denotes the indicator function of a set $S$ and for each $v\in H$ we set $l_v(u):=\langle v,u\rangle,\ u\in H$. Note that the operator $ l_v$ is linear and bounded on $H$. Combining with the stochastic Fubini theorem, see  \cite[Theo.~4.18]{DaPr}, for all $t\in[0,T]$ and $\P$-a.e.\ $\omega\in\Omega$ we have
        \begin{multline}\label{15091}
     \hskip-0.5cm\int_0^t\langle L^*(r)h,I(r,0)\rangle dr=\int_0^t l_{L^*(r)h}\Bigl(\int_0^t1_{[0,r]}(r_1)U(r,r_1)AdW(r_1)\Bigl)dr\\
       =\int_0^t \int_0^t1_{[0,r]}(r_1)l_{L^*(r)h}\bigl(U(r,r_1)A\bigl)dW(r_1)dr   \\     =\int_0^t \int_0^t1_{[0,r]}(r_1)l_{L^*(r)h}\big(U(r,r_1)A\big)dr dW(r_1)
   =\int_0^t \int_{r_1}^tl_{L^*(r)h}\big(U(r,r_1)A\big)dr dW(r_1).
    \end{multline}
    On the other hand, since $(L(t), D(L(t)))_{r_1\leq t\leq T}$ is the generator of an almost strong evolution system $(U(t,r_1))_{0\leq r_1\leq t\leq T}$ with the initial value subspace $Y$, for all $v\in Y$ we have 
     
\begin{equation}\label{ptpt}
 \int_{r_1}^tl_{L^*(r)h}U(r,r_1)vdr = \langle h,U(t,r_1)v\rangle  - \langle h,v\rangle .
\end{equation}
Since $Y$ is dense in $H$, for every $u \in G$ we can choose a sequence $(v_n)_{n\in N}\subset Y$ such that  $v_n\to Au$ as $n\to \infty$. From (\ref{ptpt}), we get
\begin{equation} \label{ptpt1}
\int_{r_1}^tl_{L^*(r)h}U(r,r_1)v_ndr = \langle h,U(t,r_1)v_n\rangle  - \langle h,v_n\rangle \qquad \text{for all}\ n\in\N.
\end{equation}
Since $\|L^*(r)h\|_H$ is bounded on $[0,T]$ and $(U(t,r_1))_{0\leq r_1\leq r\leq t}$ is a bounded family of linear bounded operators on $H$, using Lebesgue's dominated convergence theorem from (\ref{ptpt1}) we can conclude $$\int_{r_1}^tl_{L^*(r)h}U(r,r_1)Audr = \langle h,U(t,r_1)Au\rangle  - \langle h,Au\rangle \quad \text{for all}\ u\in G,$$
i.e.~$\int_{r_1}^tl_{L^*(r)h}U(r,r_1)Adr = \langle h,U(t,r_1)A\rangle  - \langle h,A\rangle $ on $G$. Back to (\ref{15091}), we obtain
          \begin{equation*}
                     \int_0^t\langle L^*(r)h,I(r,0)\rangle dr= \langle h,\int_0^tU(t,r_1)AdW(r_1)\rangle-\int_0^t\langle h, AdW(r_1)\rangle,
         \end{equation*}
i.e.~for all $h\in D(L^*)$ and $t\in[0,T]$ we have $$ \langle h,\int_0^tU(t,r)AdW(r)\rangle =  \int_0^t\langle L^*(\tau)h,\int_0^\tau U(\tau,r)AdW(r)\rangle d\tau+  \int_0^t\langle h, AdW(r_1)\rangle\quad \P{\rm -a.e.}.$$
Hence, $\int_0^tU(t,r)AdW(r)$ is an analytically weak solution.
\end{proof}


\subsection{Uniqueness}\label{subsec22}

\begin{definition}\label{07112} Assume that $(O,\langle\cdot,\cdot\rangle_O)$ is a separable Hilbert space. We call a function $\varrho: [0,T]\longrightarrow O$ is in class $C_w^1([0,T], O)$ if it is (Bochner) square integrable and 

 \noindent
 (i)~ for all $v\in O$,\ the function $[0,T]\ni r\longmapsto \langle \varrho(r),v\rangle_O\in \R$ is in $H^{1,2}((0,T),\R)$ (Sobolev space of weakly differentiable functions on $(0,T)$ which are square integrable together with their weak derivatives); 
 
 \noindent
  (ii)~ there exists a (Bochner) square integrable $\varrho'_O: [0,T]\longrightarrow O$ such that for all\ $v\in O$  $$\langle \varrho,v\rangle_O'=\langle \varrho_O',v\rangle_O\quad \text{a.e.~on } [0,T].$$
  \end{definition}
 \begin{assumption}\label{30111}
    (i) There exists an inner product $\langle\cdot,\cdot\rangle_{\D^*}$ on $\D^*$ such that $(\D^*,\langle\cdot,\cdot\rangle_{\D^*})$ is a separable Hilbert space and there exists $0<C_2<\infty$ such that $ \|\cdot\|_{D(L^*(r))}\leq C_2\|\cdot\|_{\D^*}$ for all $r\in [0,T]$, where $\|\cdot\|_{D(L^*(r))}$ is the graph norm w.r.t.~$L^*(r)$.\\    
    (ii) For all $v\in \D$ and all $0\leq t\leq T,\ \partial_\tau U(t,\tau)v=-U(t,\tau)L(\tau)v$\ for a.e.~$\tau\in[0,t]$.\\    
   (iii) There exists a subspace $Y^*\subset \D^*$, dense in $H$, such that for all $u\in Y^*, U^*(t,r)u\in \D^*$, the map $[0,T]\ni r\longmapsto  \varrho(r):=U^*(t,r)u\in \D^*$ is in $C_w^1\big([0,T],\D^*\big)$, and $\varrho_{\D^*}'(r)=-L^*(r)\varrho(r)$ in $H$ for a.e.~$r\in[0,T]$.      
 \end{assumption} 
 \begin{remark}
     By Assumption \ref{30111}(i), for all $u\in \D^*$ we have $$ \sup_{t\in[0,T]}\|L^*(t)u\|_H\leq \sup_{t\in[0,T]}\|u\|_{D(L^*(t))}\leq C_2\|u\|_{\D^*}.$$ Hence, $[0,T]\ni t\longmapsto \|L^*(t)u\|_H$  is bounded for each $u\in \D^*$. This is the assumption in Theorem \ref{existence}(ii).
\end{remark}
    \begin{theorem} \label{main} Let Assumption \ref{30111} hold. Then the analytically weak solution of (\ref{SDE}) is unique.
\end{theorem}
 
 Before we can prove Theorem \ref{main}, we need the following proposition.
\begin{proposition} \label{pm} Let $(X(t))_{0\leq t\leq T}$ be an analytically weak solution of (\ref{SDE}), then for all $\varrho\in C_w^1([0,T], \D^*)$ and all $t\in[0,T]$ we have
    \begin{equation}\label{ptlemma}
        \left\langle X(t), \varrho(t)\right\rangle = \int_0^t\left\langle X(r), \varrho'_{\D^*}(r)+L^*(r)\varrho(r)\right\rangle dr + \int_0^t\left\langle A dW(r),\varrho(r)\right\rangle  \quad \P{\rm -a.e.}, 
    \end{equation}
where $\varrho'_{\D^*}$ is as in Definition \ref{07112} (ii).
\end{proposition}
Denote by $\left\langle A(\cdot),\varrho(r)\right\rangle$ the family of continuous linear functionals $J_r : G \to \R$, $x \mapsto \left\langle A x,\varrho(r)\right\rangle$. Then the stochastic integral in \eqref{ptlemma} is meant in the sense of Remark \ref{RemStochasticIntegral}, i.e., $\int_0^t\left\langle A dW(r),\varrho(r)\right\rangle = \int_0^t \left\langle A (\cdot),\varrho(r)\right\rangle dW(r)$.
\begin{proof}
We consider the case $\varrho=\Phi u$, where $\Phi\in C^1([0,T],\R)$ and $u\in \D^*$. Note that $\varrho:[0,T]\longrightarrow \D^*$ is continuously differentiable and its derivative is $\Phi'u$. Let $(X(t))_{0\leq t\leq T}$ be an analytically weak solution to (\ref{SDE}). For all $t\in[0,T]$, we define $$F_u(t):=\int_0^t\langle X(r),L^*(r)u\rangle dr +\langle AW(t),u\rangle.$$ Then $\langle X(t),\varrho(t)\rangle=\Phi(t)\langle X(t),u\rangle=\Phi(t)F_u(t)$ for all $t\in[0,T]$. Apply It{\^o}'s formula to the process  $(\Phi(t)F_u(t))_{0\leq t\leq T}$ and obtain
 \begin{equation}\label{1709a}
  \langle X(t),\Phi(t)u\rangle=\int_0^t\Phi(r)\langle AdW(r),u\rangle
  +\int_0^t\Phi(r)\langle X(r),L^*(r)u\rangle+\Phi'(r)\langle X(r),u\rangle dr.
 \end{equation} 
Equation (\ref{1709a}) is also satisfied for $\Phi\in H^{1,2}((0,T),\R).$ Indeed, let $\Phi\in H^{1,2}(0,T)$, then there exists a sequence $(\Phi_n)_{n\in\N}\subset C^1([0,T])$ such that $\lim_{n\to\infty}\Phi_n=\Phi\ \text{in}\ C([0,T])$ and $\lim_{n\to\infty}\Phi'_n=\Phi'\ \text{in}\ L^2(0,T).$ Using (\ref{1709a}) for $\Phi_n, n\in\N$, we have
 \begin{equation}\label{1709b}
  \hskip-0.1cm\langle X(t),\Phi_n(t)u\rangle=\int_0^t\Phi_n(r)\langle AdW(r),u\rangle+\int_0^t\Phi_n(r)\langle X(r),L^*(r)u\rangle+\Phi_n'(r)\langle X(r),u\rangle dr.
 \end{equation} 
 We need to pass a limit in (\ref{1709b}) as $n\to\infty.$ First, for all $r\in[0,t], u\in \D^*$, we have 
 \begin{equation*}
  |\Phi_n(r)\langle X(r),L^*(r)u\rangle|\leq |\Phi_n(r)|\|X(r)\|\|L^*(r)u\|.
  \end{equation*}  
  Since for each $u\in \D^*$ the map $[0,T]\ni r\longmapsto \|L^*(r)u\|$ is bounded, the analytically weak solution of (\ref{SDE}) has $\P$-a.e.~square integrable paths, and $(\Phi_n)_{n\in\N}$ converges to $\Phi$ in $L^2(0,T)$, we have 
  \begin{equation}\label{17093}
  \lim_{n\to\infty}\int_0^t\Phi_n(r)\langle X(r),L^*(r)u\rangle dr=\int_0^t\Phi(r)\langle X(r),L^*(r)u\rangle dr.
  \end{equation}  
  Second, we obtain 
  \begin{equation}\label{17094}
  \lim_{n\to\infty}\int_0^t\Phi'_n(r)\langle X(r),u\rangle dr=\int_0^t\Phi'(r)\langle X(r),u\rangle dr
  \end{equation} by combining convergence of $(\Phi'_n)_{n\in\N}$ to $\Phi'$ in $L^2(0,t)$ with square integrability of the map $[0,T]\ni t\mapsto X(t)$. Third, we prove $\lim_{n\to\infty}\int_0^t\Phi_n(r)\langle AdW(r),u\rangle=\int_0^t\Phi(r)\langle AdW(r),u\rangle.$ Note that, by Sobolev embedding, $\lim_{n\to\infty}\sup_{r\in[0,T]}|\Phi_n(r)-\Phi(r)|=0$. Using It{\^o}'s isometry, see \cite[Prop.~4.5 and p.~94]{DaPr}, we get
  \begin{equation*}
    \begin{split}    \E\Big|\int_0^t(\Phi_n(r)-\Phi(r))\langle AdW(r),u\rangle\Big|^2&=\int_0^t|\Phi_n(r)-\Phi(r)|^2\|\langle AQ^{\frac{1}{2}}\cdot,u\rangle\|^2_{L_2(G,\R)}dr\\
    &\leq t\sup_{r\in[0,t]}|\Phi_n(r)-\Phi(r)|^2\|AQ^{\frac{1}{2}}\|^2_{L(G,H)}\|u\|^2_H\xrightarrow{n\to\infty} 0.
    \end{split}
  \end{equation*}
Hence, there exists a subsequence $(\Phi_{n_k})_{k\in\N}$  such that $$\lim_{k\to\infty}\int_0^t(\Phi_{n_k}(r)-\Phi(r))\langle AdW(r),u\rangle=0,\ \P{\rm -a.e.}.$$ 
 Combining (\ref{17093}) with (\ref{17094}) and passing to the limit as $k\to\infty$ in (\ref{1709b}) for the subsequence $(\Phi_{n_k})_{k\in\N}$ we have
  \begin{equation}\label{0512}
  \langle X(t),\Phi(t)u\rangle=\int_0^t\Phi(r)\langle AdW(r),u\rangle
  +\int_0^t\big(\Phi(r)\langle X(r),L^*(r)u\rangle+\Phi'(r)\langle X(r),u\rangle \big)dr\ \ \P-{\rm a.e.},
 \end{equation} 
 i.e.~(\ref{ptlemma}) is shown for the case $\varrho=\Phi u$.
  
 Now let $\varrho\in C_w^1$ be general. Since $\big(\D^*,\langle\cdot,\cdot\rangle_{\D^*}\big)$ is a separable Hilbert space, we can choose an the orthonormal basis $\{e_k |~ k\in\N\}$ of $\D^*$. Since $\varrho(r)$ and $\varrho_{\D^*}'(r)$ are in $\D^*$ for a.e.~$r\in[0,T]$ we have $\varrho=\sum_{k=1}^\infty \left\langle \varrho,e_k\right\rangle_{\D^*} e_k\quad \text{and}\quad \varrho'_{\D^*}=\sum_{k=1}^\infty {\left\langle \varrho'_{\D^*},e_k\right\rangle_{\D^*} e_k}.$ Set $\varrho_{\D^*}^{(n)}=\sum_{k=1}^n \left\langle \varrho,e_k\right\rangle_{\D^*} e_k$\ and\ $\varrho'^{(n)}_{\D^*}=\sum_{k=1}^n \left\langle \varrho'_{\D^*},e_k\right\rangle_{\D^*} e_k.$ We have $\lim_{n\to\infty}\varrho_{\D^*}^{(n)}(r)= \varrho(r)\ \text{ and}\ \lim_{n\to\infty}\varrho'^{(n)}_{\D^*}(r) = \varrho'_{\D^*}(r)$ w.r.t.~$\|\cdot\|_{\D^*}$ for a.e.~$r\in[0,T]$. By (\ref{0512}), the linearity of the inner product, and the integrands together with \ref{07112}(ii), (\ref{ptlemma}) is also satisfied for the case $\varrho(r)=\varrho_{\D^*}^{(n)}(r)$, i.e., for  $\P$-a.e.~we have 
                \begin{equation} \label{ptlemma1}
                    \big\langle X(t), \varrho_{\D^*}^{(n)}(t)\big\rangle_H = \int_0^t\big\langle X(r), \varrho'^{(n)}_{\D^*}(r)+L^*(r)\varrho_{\D^*}^{(n)}(r)\big\rangle_H dr + \int_0^t\big\langle AdW(r),\varrho_{\D^*}^{(n)}(r)\big\rangle_H.
             \end{equation}
   We need to pass to the limit of (\ref{ptlemma1}) as $n\to\infty$. By It{\^o}'s isometry we have 
   \begin{multline}\label{16b}
  \E\Big\|\int_0^t\big\langle \varrho_{\D^*}^{(n)}(r)-\varrho(r),AdW(r)\big\rangle\Big\|^2 = \int_0^t\big\|\big\langle \varrho_{\D^*}^{(n)}(r)-\varrho(r),AQ^{\frac{1}{2}}\big\rangle\big\|^2_{L_2(G,\R)}dr \\
\leq\|AQ^{\frac{1}{2}}\|^2_{L(G,H)}\int_0^t\big\|\varrho_{\D^*}^{(n)}(r)-\varrho(r)\big\|^2_Hdr.
\end{multline} 
Since $(\varrho_{\D^*}^{(n)}(r))_{n\in\N}$ converges to $\varrho(r)$ w.r.t.~$\|\cdot\|_{\D^*}$, which is stronger than $\|\cdot\|_H$,
together with the integrability of $\|\varrho\|^2_{\D^*}\geq \|\varrho_{\D^*}^{(n)}\|^2_{\D^*}\geq \|\varrho_{\D^*}^{(n)}\|^2_H$,
we obtain that the estimator in (\ref{16b}) converges to zero as $n \to \infty$. Hence we can find a subsequence
$(\varrho_{\D^*}^{(n_k)}(r))_{k\in\N}$ of $(\varrho_{\D^*}^{(n)}(r))_{n\in\N}$ such that
$$\int_0^t\big\langle \varrho_{\D^*}^{(n_k)}(r),AdW(r)\big\rangle \longrightarrow \int_0^t\big\langle \varrho(r),AdW(r)\big\rangle\quad \P-\text{a.e.\ as}\ k\to\infty.$$   
We denote $\varrho_{\D^*}^{(n_k)}$ again by $\varrho_{\D^*}^{(n)}$. Since $\lim_{n\to\infty}\varrho_{\D^*}^{(n)}= \varrho$\ and\ $\lim_{n\to\infty}\varrho'^{(n)}_{\D^*}= \varrho'_{\D^*}$ w.r.t.~$\|\cdot\|_{\D^*}$, we have $\big\langle X(t),\varrho_{\D^*}^{(n)}(t)\big\rangle \rightarrow \big\langle X(t),\varrho(t)\big\rangle$\ and\ $\big\langle X(r),\varrho'^{(n)}_{\D^*}(r)\big\rangle \rightarrow \big\langle X(r),\varrho'_{\D^*}(r)\big\rangle$ as  $n\to\infty$ for a.e.~$r\in[0,T]$. Moreover, since $\varrho\in C_w^1([0,T], \D^*)$, for a.e.~$ r\in[0,T]$ we have $\big|\big\langle X(r),\varrho'^{(n)}_{\D^*}(r)\big\rangle\big|\leq \|X(r)\| \|\varrho'^{(n)}_{\D^*}(r)\|\leq C\|X(r)\| \|\varrho'^{(n)}_{\D^*}(r)\|_{\D^*} \leq C\|X(r)\|\|\varrho'_{\D^*}(r)\|_{\D^*}$. As before, the map $[0,T]\ni r\mapsto\|X(r)\|^2\in\R$  is integrable and by \ref{07112}(ii) also the map $[0,T]\ni r\mapsto \|\varrho'_{\D^*}(r)\|^2_{\D^*}\in\R$\ is integrable. Hence, the map $[0,t]\ni r\mapsto\|X(r)\|\|\varrho'_{\D^*}(r)\|_{\D^*}\in\R$ is integrable for all $t\in[0,T]$ and by the dominated convergence theorem we get $$\int_{0}^t\left\langle X(r),\varrho'^{(n)}_{\D^*}(r)\right\rangle dr \longrightarrow \int_{0}^t\left\langle X(r),\varrho'_{\D^*}(r)\right\rangle dr\quad\text{as}\ n\to\infty.$$    
   Further we need to check that $\lim_{n\to\infty}\int_{0}^t\big\langle X(r),L^*(r)\varrho_{\D^*}^{(n)}(r)\big\rangle dr = \int_{0}^t\left\langle X(r),L^*(r)\varrho(r)\right\rangle dr.$ For a.e.~$r\in[0,T]$ we have
      \begin{equation*}
           \begin{split}
              \big|\big\langle X(r),L^*(r)\varrho_{\D^*}^{(n)}(r)\big\rangle\big|&\leq \|X(r)\| \|L^*(r)\varrho_{\D^*}^{(n)}(r)\|\leq \|X(r)\| \|\varrho_{\D^*}^{(n)}(r)\|_{D(L^*(r))}\\
                                                         &\leq C_2 \|X(r)\| \|\varrho_{\D^*}^{(n)}(r)\|_{\D^*}\leq C_2\|X(r)\|\|\varrho(r)\|_{\D^*}.
           \end{split}
    \end{equation*}     
      By assumption, the maps $[0,T]\ni r\longmapsto \|\varrho(r)\|_{\D^*}^2\in\R$ and $[0,T]\ni r\longmapsto \|X(r)\|^2\in\R$ are integrable, hence also $[0,T]\ni r\longmapsto \|X(r)\|\|\varrho(r)\|_{\D^*}\in\R$ is integrable.
      Moreover, 
       \begin{equation*}
           \begin{split}
      \|L^*(r)\varrho_n(r)-L^*(r)\varrho(r)\|&=\|L^*(r)(\varrho_n(r)-\varrho(r))\|\\
                                       &\leq\ \|\varrho_n(r)-\varrho(r)\|_{D(L^*(r))}\leq\ C_2\|\varrho_n(r)-\varrho(r)\|_{\D^*} \xrightarrow{n\to\infty} 0.
         \end{split}
    \end{equation*}
     Hence, by the dominated convergence theorem we have $$\int_0^t\left\langle X(r),L^*(r)\varrho_n(r)\right\rangle dr\longrightarrow \int_0^t\left\langle X(r),L^*(r)\varrho(r)\right\rangle dr\quad\ \text{as}\ n\to\infty. $$
    
 Finally, taking the limit as $n\to\infty$ in (\ref{ptlemma1}), we obtain
     \begin{equation*}
        \left\langle X(t), \varrho(t)\right\rangle = \int_0^t\left\langle X(r), \varrho'_{\D^*}(r)+L^*(r)\varrho(r)\right\rangle dr + \int_0^t\left\langle AdW(r),\varrho(r)\right\rangle\quad \P{\rm -a.e.}.\qedhere
    \end{equation*}  
\end{proof}   
\begin{proof} [Proof of Theorem \ref{main}] Let $0\leq t\leq T$. Consider the map $[0,t]\ni r\longmapsto\varrho(r):=U^*(t,r)u\in \D^*,$ where $u\in Y^*$. By assumption, $\varrho\in C_w^1([0,t],\D^*)$ and $\varrho_{\D^*}'(r)= -L^*(r)\varrho(r)$\ in $H$ for a.e.~$r\in[0,T]$. Applying Proposition \ref{pm} to the case $\varrho(r)= U^*(t,r)u,\ u\in Y^*$, we have
$$\left\langle X(t),u\right\rangle = \int_{0}^t\Big\langle AdW(r),U^*(t,r)u\Big\rangle = \Big\langle \int_{0}^tU(t,r)AdW(r),u\Big\rangle \ \P{\rm -a.e.}.$$
Since $Y^*$ is a dense subspace of the separable Hilbert space H, we get $$X(t)=\int_{0}^tU(t,r)AdW(r)\quad \P{\rm -a.e.}.$$  The analytically weak solution is unique.
 \end{proof} 

   \section{Application: a SPDE from industrial mathematics}\label{apply}

Recall the stochastic partial differential equation \eqref{pdae} from the introduction:
\begin{multline}\label{ptae}
d \, \partial_t \x(s,t) = \big(\partial_s(\lambda\partial_s\x)(s,t) - b\partial_{ssss}\x(s,t) \\
-g\e_3+\f^{{\rm det}}(s,t)\big)dt+\sigma d\w(s,t), \quad (s,t)\in[0,l]\times [0,T],\tag{\ref{pdae}}
\end{multline}
with initial condition
\begin{equation}\label{inicont}
\x(s,0) = (s-l)\e_3,\quad \partial_t\x(s,0)=\0,\quad s \in[0,l], \tag{\ref{pdae}a}
\end{equation}
boundary condition
  \begin{equation}\label{boundcont}
    \x(l,t)=\0,\quad
    \partial_s\x(l,t)=\e_3,\quad
    \partial_{ss}\x(0,t)=\0,\quad
     \partial_{sss}\x(0,t)=\0,\quad t\in[0,T].\tag{\ref{pdae}b}
  \end{equation}

Let $X(t):=
   \dbinom{\x(t)}{\partial_t\x(t)}$, then 
 \begin{equation*}
   \begin{split}
   dX(t)=
   \dbinom{
   d\x(t)}{
   \partial_t \partial_t\x(t)}&=
   \begin{pmatrix}
     \partial_t\x(t)dt\\
      \big(\partial_s(\lambda(t)\partial_s\x(t)) - b\partial_{ssss}\x(t)+\f(t)\big)dt+d\w(t)
   \end{pmatrix}\\
   &=
   \Bigg(\begin{pmatrix}
   0& Id\\
   \partial_s(\lambda(t)\partial_s)-b\partial_{ssss}& 0
   \end{pmatrix}X(t)+\dbinom{0}{\f(t)}\Bigg)dt +\sigma d\dbinom{0}{\w(t) }.
  \end{split} 
 \end{equation*}  
 Hence (\ref{ptae}) becomes 
 \begin{equation}\label{SDE1}
 dX(t)=\big(L(t)X(t)+F(t)\big)dt + AdW(t),
 \end{equation}
 where
 \begin{equation*}
    L(t)= \begin{pmatrix}
   0& Id\\
   \tilde{L}(t)& 0
   \end{pmatrix},\ F(t)=\dbinom{0}{\f(t)},\  A=\sigma \begin{pmatrix}
   0& 0\\
   0& Id
   \end{pmatrix},\ W(t) = \dbinom{0}{\w(t)},   
 \end{equation*}   
 $\tilde{L}(t)=\partial_s(\lambda(t)\partial_s)-b\partial_{ssss} $ and $\f(t)=-g\e_3+\f^{{\rm det}}(t),\ 0\leq t\leq T$.

\subsection{The homogeneous problem}\label{homogen}

First we consider \eqref{ptae} with general initial condition
\begin{equation}\label{inicontt}
\x(s,0) = \boldsymbol{\xi}_1(s),\quad \partial_t\x(s,0)=\boldsymbol{\xi}_2(s),\quad s \in[0,l], \tag{\ref{ptae}g}
\end{equation}
and homogeneous boundary condition
  \begin{equation}\label{boundcontt}
    \x(l,t)=\0,\quad
    \partial_s\x(l,t)=\0,\quad
    \partial_{ss}\x(0,t)=\0,\quad
     \partial_{sss}\x(0,t)=\0,\quad t\in[0,T].\tag{\ref{ptae}h}
  \end{equation}
	We use the notation $L^2(0,l):= L^2((0,l);\R^3)$ (space of $\R^3$-valued functions, square integrable w.r.t.~the Lebesgue measure on $(0,l)$,
	equipped with its usual inner product) and consider $\tilde{L}(t): D(\tilde{L}(t))\subset L^2(0,l)\longrightarrow L^2(0,l)$
	with domain $$D(\tilde{L}(t)) := \big\{{\bf \v}\in H^{4,2}((0,l);\R^3) \big| \text{\ boundary conditions in (\ref{boundcontt}) are fulfilled}\big\}=: H_{{\rm bc}}^{4,2}(0,l)$$
	($H^{m,2}((0,l);\R^3)$ denotes the Hilbert space of $\R^3$-valued, $m$ times weakly differentiable functions on $(0,l)$
	which are square integrable together with their weak derivatives). Those domains are independent of $t\in[0,T]$. Let $$H_{{\rm bc}}^{2,2}(0,l)
	:=\big\{\u\in H^{2,2}((0,l);\R^3)\big|\text{\ first two boundary conditions in (\ref{boundcontt}) are fulfilled}\big\}$$ with inner product defined by $\langle \u,\v\rangle_{H_{{\rm bc}}^{2,2}(0,l)}:= b\int_0^l\langle\partial_{ss}\u,\partial_{ss}\v\rangle_{{\rm euk}} ds,\ \u,\v\in H_{{\rm bc}}^{2,2}(0,l).$ We consider the Hilbert space $H:=H_{{\rm bc}}^{2,2}(0,l)\times L^2(0,l)$ with inner product  $$\left\langle {\dbinom{\u_1}{\v_1},\dbinom{\u_2}{\v_2}}\right\rangle_H := \langle{\u_1,\u_2}\rangle_{H_{{\rm bc}}^{2,2}(0,l)}+\langle {\v_1,\v_2}\rangle_{L^2(0,l)},\quad \dbinom{\u_1}{\v_1},\dbinom{\u_2}{\v_2}\in H,$$ and the family of operators 
\begin{equation*}
  \begin{split}
     L(t): D(L(t))\subset H& \longrightarrow H\\
                                \dbinom{\u}{\v}& \longmapsto \dbinom{\v}{\tilde{L}(t)\u},
  \end{split}
\end{equation*}
   where the domains $D(L(t)):=H_{{\rm bc}}^{4,2}(0,l)\times H_{{\rm bc}}^{2,2}(0,l)=:\D$ are independent of $t\in[0,T]$. The inner product on $\D$ we define by 
   \begin{equation}\label{10114}
   \left\langle {\dbinom{\u_1}{\v_1},\dbinom{\u_2}{\v_2}}\right\rangle_\D := \langle{\u_1,\u_2}\rangle_{H_{{\rm bc}}^{4,2}(0,l)}+\langle {\v_1,\v_2}\rangle_{H_{{\rm bc}}^{2,2}(0,l)},\ \dbinom{\u_1}{\v_1},\dbinom{\u_2}{\v_2}\in \D,
 \end{equation}  
 where $\langle{\u_1,\u_2}\rangle_{H_{{\rm bc}}^{4,2}(0,l)}:=b^2\int_0^l\langle\partial_{ssss}\u_1,\partial_{ssss}\u_2\rangle_{{\rm euk}}ds$\ for all $\u_1,\u_2\in H^{4,2}_{{\rm bc}}(0,l)$. Of course, $(\D,\langle\cdot,\cdot\rangle_\D)$ is a separable Hilbert space with norm $\|\cdot\|_\D:=\sqrt{\langle\cdot,\cdot\rangle_\D}$. Furthermore, we also use
\begin{multline*}
H_{{\rm bc}}^{6,2}(0,l):=\big\{\u\in H^{6,2}((0,l);\R^3)\big| \\ \text{boundary conditions in \eqref{boundcontt}} \
\text{and}\ \partial_{ssss}\u(l)= \partial_{sssss}\u(l)=\0\ \text{hold}\big\}.
\end{multline*}

\begin{assumption}\label{ass}
	(i) $\lambda(t)\in H^{3,2}((0,l);\R)$ for all $t\in[0,T]$,\  $\sup_{t\in[0,T]}\|\lambda(t)\|_{H^{3,2}((0,l);\R)}<\infty$,
	$\lambda(0,t)=\lambda(l,t)=\partial_s\lambda(l,t)=\partial_s\lambda(0,t)=0,\ \lambda(s,t)>0$ for all $s\in(0,l),\ t\in[0,T]$,
	and $\lambda,$  $\partial_s\lambda$ are measurable on $[0,l]\times[0,T]$.

\noindent
(ii) The map $[0,T]\ni t\mapsto \f^{{\rm det}}(t)\in L^2(0,l)$ is Bochner integrable.

\noindent
(iii) $(\w(t))_{0\leq t\leq T}$ is an $L^2(0,l)$-valued $Q$-Wiener process with ${\rm Tr}(Q)<\infty$.

\end{assumption}

\begin{theorem} \label{12101} Let Assumption \ref{ass} hold. Then there exists a unique analytically weak solution to \eqref{ptae},
\eqref{inicontt}, \eqref{boundcontt} for all\ $\boldsymbol{\xi}_1 \in H_{{\rm bc}}^{6,2}(0,l)$ and\ $\boldsymbol{\xi}_2 \in H_{{\rm bc}}^{4,2}(0,l)$.
 \end{theorem} 
Before we can prove Theorem \ref{12101} we need several lemmas and propositions to construct an almost strong evolution system having sufficient properties to apply our result from Sections \ref{subsec21} and \ref{subsec22}. We decompose $L(t)$ as follows
  $$L(t)= \begin{pmatrix}
   0& Id\\
   -b\partial_{ssss}& 0
   \end{pmatrix}+ \begin{pmatrix}
   0& 0\\
   \partial_s(\lambda(t)\partial_s)& 0
   \end{pmatrix}=:L_0+L_1(t),\quad 0\leq t\leq T.$$
 \begin{lemma}\label{lm1}
Let Assumption \ref{ass}(i) hold. Then we have:\\
(i) The operator $L_1(t)$  is in $L(H)$ and $L(\D)$ for every\ $t\in[0,T]$ and there exist $0<C_4,C_5<\infty$ such that  $\sup_{t\in[0,T]}\|L_1(t)\|_{L(H)}\leq C_4$ and $\sup_{t\in[0,T]}\|L_1(t)\|_{L(\D)}\leq C_5$.\\
(ii) The operator $L(t)$  is in $L(\D,H)$ for every\ $t\in[0,T]$ and the map \ $[0,T]\ni t\longmapsto L(t)\in L(\D,H)$ is bounded and measurable.
 \end{lemma}
 \begin{proof}
 (i): First, by using Sobolev embedding for $\lambda$ as in Assumption \ref{ass} we get that $\lambda$ has a twice continuously differentiable version on $[0,l]$. Since $\u\in H^{2,2}_{{\rm bc}}(0,l)$ we have that also $\partial_s u$ has a absolutely continuous version with weak derivative in $L^2(0,l)$. Using the fundamental theorem of calculus for absolutely continuous functions we get 
\begin{equation}\label{07091}
\sup_{s\in[0,l]}|\lambda(s,t)|^2\leq l\int_0^l(\partial_s\lambda(s,t))^2ds\quad
\text{ and}\quad \sup_{s\in[0,l]}\| \partial_s\u(s)\|_{{\rm euk}}^2\leq l\int_0^l\|\partial_{ss}\u(s)\|_{{\rm euk}}^2ds.
\end{equation}
Hence, for all $\dbinom{ \u}{\v}\in H$ we have by (\ref{07091}) together with Assumption \ref{ass}(i)
\begin{equation}\label{19081}
\begin{split}
\Big\|L_1(t)\dbinom{\u}{\v}\Big\|^2_H&=\|\partial_s(\lambda(t)\partial_s\u)\|^2_{L^2(0,l)}=\int_0^l\|\partial_s\lambda(s,t)\partial_s\u(s)+\lambda(s,t)\partial_{ss}\u(s)\|_{{\rm euk}}^2ds\\
&\hskip-1.1cm\leq 2\int_0^l(\partial_s\lambda(s,t))^2\|\partial_s\u(s)\|_{{\rm euk}}^2 ds+2\int_0^l(\lambda(s,t))^2\|\partial_{ss}\u(s)\|_{{\rm euk}}^2ds\\
&\hskip-1cm\leq4l\int_0^l(\partial_s\lambda(s,t))^2ds\int_0^l\|\partial_{ss}\u(s)\|_{{\rm euk}}^2ds \leq C_3\|\u\|^2_{H^{2,2}_{{\rm bc}}(0,l)} \leq C_4\Big\|\dbinom{\u}{\v}\Big\|_H^2
\end{split} 
\end{equation}
for some $0< C_4<\infty$\  independent of\ $t\in[0,T].$
Hence, $\sup_{t\in[0,T]}\|L_1(t)\|_{L(H)}\leq C_4.$

Second, by Assumption \ref{ass}(i), $L_1(t)(\D)\subset \D$ for all $t\in[0,T]$. Similarly to (\ref{19081}), we can prove that there exists $0<C_5<\infty$ such that $\|L_1(t)w\|_\D\leq C_5\|w\|_\D$ for all $t\in[0,T]$ and $w\in\D$, i.e., $\sup_{t\in[0,T]}\|L_1(t)\|_{L(\D)}\leq C_5$.

(ii): Since $Id: \D\longrightarrow H$ is continuous, together with Lemma \ref{lm1}(i) we obtain that $\sup_{t\in[0,T]}\|L(t)\|_{L(\D,H)}<\infty.$ Measurability of the map \ $[0,T]\ni t\longmapsto L(t)\in L(\D,H)$ follows by measurability of $\lambda$ and $\partial_s\lambda$.
 \end{proof} 
 \begin{lemma}\label{lm2}
 The operator $L_0:\D\subset H\longrightarrow H$ is closed, skew-adjoint, $L_0$ and $L_0^*$ are dissipative, and (hence) $(L_0,\D)$ generates a $C_0$-semigroup of contractions.
 \end{lemma}
 \begin{proof}
 To prove that $(L_0,\D)$ is closed, we choose arbitrary $\u_n\rightarrow \u$ in $H_{{\rm bc}}^{2,2}(0,l)$ and $\v_n\rightarrow \v$ in $L^2(0,l)$ such that $\v_n\longrightarrow \y_1$ in $H_{{\rm bc}}^{2,2}(0,l)$ and $-b\partial_{ssss}\u_n\rightarrow \y_2$ in $L^2(0,l)$ as $n\to\infty$.
 
  Since $\v_n\rightarrow \y_1$ in $H_{{\rm bc}}^{2,2}(0,l)$ and the norm $\|\cdot\|_{H^{2,2}_{{\rm bc}}(0,l)}$ is stronger than the norm $\|\cdot\|_{L^2(0,l)}$ then $\v_n\rightarrow \y_1$ in $L^2(0,l)$ as $n\to\infty$. Hence, $\v=\y_1\in H^{2,2}_{{\rm bc}}(0,l)$. Define $\tilde{H}^{2,2}_{{\rm bc}}(0,l):=\big\{ \u\in H^{2,2}((0,l);\R^3)\big|\ \u(0)=\partial_s\u(0)=\0\big\}$ and the norm on $\tilde{H}^{2,2}_{{\rm bc}}(0,l)$ is as on $H^{2,2}_{{\rm bc}}(0,l)$. Since $\u_n\rightarrow \u$ in $H_{{\rm bc}}^{2,2}(0,l)$, we have $\partial_{ss}\u_n\rightarrow \partial_{ss}\u$ in $L^2(0,l)$ as $n\to\infty$. Moreover, since $\y_2\in L^2(0,l)$  we have $\z(s):=\int_0^s\int_0^{s_1}\y_2(s_2)ds_2ds_1\in \tilde{H}^{2,2}_{{\rm bc}}(0,l)$. Because $-b\partial_{ssss}\u_n\rightarrow \y_2$ in $L^2(0,l)$, $-b\partial_{ss}\u_n\rightarrow \z$ in $\tilde{H}^{2,2}_{{\rm bc}}(0,l)$ as $n\to\infty$. Since $\partial_{ss}\u_n\rightarrow \partial_{ss}\u$ in $L^2(0,l)$ as $n\to\infty$, we have $-b\partial_{ss}\u=\z\in \tilde{H}^{2,2}_{{\rm bc}}(0,l)$. Hence, $\u\in H^{4,2}_{{\rm bc}}(0,l)$ and $-b\partial_{ssss}\u=\partial_{ss}\z=\y_2.$ Combining with $\v=\y_1\in H_{{\rm bc}}^{2,2}(0,l)$, we can conclude that $(L_0,\D)$ is closed.
  
  To prove $(L_0,\D)$ is skew-adjoint, we need to obtain that $L_0$ is skew-symmetric and $D(L_0^*)=D(L_0)=\D$. Indeed, one can check easily that $L_0^*\big|_{\D}=-L_0$, i.e.~$L_0$ is skew-symmetric.
  
  Next, we shall find $\dbinom{ \u_1}{\v_1}, \dbinom{ \u_2}{\v_2}\in H$ such that for all $\dbinom{ \u}{\v}\in \D,$
\begin{equation}\label{ss} \left\langle L_0\dbinom{ \u}{\v}, \dbinom{ \u_1}{\v_1}\right\rangle_H=\left\langle\dbinom{ \u}{\v},\dbinom{ \u_2}{\v_2}\right\rangle_H.
\end{equation}
Equation (\ref{ss}) is equivalent to
\begin{equation}\label{ss1} 
 b\int_0^l \langle\partial_{ss}\v,\partial_{ss}\u_1\rangle_{{\rm euk}}ds-b\int_0^l \langle\partial_{ssss}\u,\v_1\rangle_{{\rm euk}}ds=b\int_0^l \langle\partial_{ss}\u,\partial_{ss}\u_2\rangle_{{\rm euk}}ds+\int_0^l \langle\v,\v_2\rangle_{{\rm euk}}ds
\end{equation}
for all\ $\dbinom{ \u}{\v}\in \D.$
For $\u=\0$, from (\ref{ss1}) we have 
\begin{equation}\label{ss2}
b\int_0^l \langle\partial_{ss}\v,\partial_{ss}\u_1\rangle_{{\rm euk}}ds=\int_0^l \langle\v,\v_2\rangle_{{\rm euk}}ds \quad \text{for all}\ \v\in H^{2,2}_{{\rm bc}}(0,l),
\end{equation}
hence, in particular, for all $\v\in C^\infty_c(0,l)$ (space of $\R^3$-valued $C^\infty$ functions with compact support in $(0,l)$).
Since $\v_2\in L^2(0,l)$ we have that $\partial_{ss}\u_1$ is continuously differentiable with $\partial_{sss}\u_1\in L^2(0,l)$ and $\partial_{sss}\u_1$ is a.e.~differentiable with $\partial_{ssss}\u_1\in L^2(0,l)$, see \cite[Theo.~2.4.2]{Mik}. Hence, $\u_1\in H^{4,2}((0,l);\R^3)$. We check now the boundary conditions of $\u_1$. Since\ $\partial_{s}\v(l)=\v(0)=\0$ and \ $b\partial_{ssss}\u_1=\v_2$\ a.e.~on $(0,l)$, two integration by parts yield
\begin{multline*}
 \hskip-0.5cm b\int_0^l \langle\partial_{ss}\v,\partial_{ss}\u_1\rangle_{{\rm euk}}ds=b\Big(\langle\partial_{s}\v(0),\partial_{ss}\u_1(0)\rangle_{{\rm euk}}-\langle\v(l),\partial_{sss}\u_1(l)\rangle_{{\rm euk}}+\int_0^l \langle\v,\partial_{ssss}\u_1\rangle_{{\rm euk}}ds\Big)\\               =b\big(\langle\partial_{s}\v(0),\partial_{ss}\u_1(0)\rangle_{{\rm euk}}-\langle\v(l),\partial_{sss}\u_1(l)\rangle_{{\rm euk}}\big)+\int_0^l \langle\v,\v_2\rangle_{{\rm euk}}ds,
  \end{multline*}
  for all $\v\in H_{{\rm bc}}^{2,2}(0,l)$. Comparing with (\ref{ss2}), for arbitrary $\v\in H_{{\rm bc}}^{2,2}(0,l)$ we obtain $\langle\partial_{s}\v(0),\partial_{ss}\u_1(0)\rangle_{{\rm euk}}-\langle\v(l),\partial_{sss}\u_1(l)\rangle_{{\rm euk}}=\0$. That implies $\partial_{ss}\u_1(0)=\partial_{sss}\u_1(l)=\0$, i.e.~$\u_1\in H_{{\rm bc}}^{4,2}(0,l).$
Similarly, we can identify $\v_1\in H_{{\rm bc}}^{2,2}(0,l).$ So, $D(L_0^*)\subset H_{{\rm bc}}^{4,2}(0,l)\times H_{{\rm bc}}^{2,2}(0,l)=D(L_0)$. We already know that  $L_0$ is skew-symmetric. Thus, $L_0$ is even skew-adjoint.

Clearly, $L_0$ and $L_0^*=-L_0$ are dissipative. Due to \cite[Corol.~1.4.4]{Pazy}, both $(L_0,\D)$ and $(L_0^*,\D)$ are generators of contraction semigroups.
 \end{proof}
 \begin{lemma} \label{lm3} Let Assumption \ref{ass}(i) hold, then on $\D$ is $\|\cdot\|_\D=\|L_0\cdot\|_H$. Moreover, there exist $0<c_6, C_6<\infty$ such that $$c_6\|\cdot\|_\D\leq\|\cdot\|_{D(L(t))} \leq C_6 \|\cdot\|_\D\quad \text{for all}\ t\in[0,T].$$
 \end{lemma}  
 \begin{proof} We can check that $\|\cdot\|_\D=\|L_0\cdot\|_H$ on $\D$ just by using definition of the norms and $L_0$. Combining with Lemma \ref{lm1}, for all $w\in\D$ and $t\in[0,T]$ we have
  $$\|w\|_{D(L(t))}=\|w\|+\|L(t)w\|\leq \|w\|+\|L_1(t)w\|+\|L_0w\|\leq (1+C_4)\|w\|+\|w\|_\D\leq C_6\|w\|_\D,$$
 for some $0<C_6<\infty$ and $$\|w\|_\D=\|L_0w\|\leq \|L_1(t)w\|+\|L(t)w\|\leq C_4\|w\|+\|L(t)w\|\leq (1+C_4)C_6\|w\|_{D(L(t))}.$$
  Hence, $c_6\|\cdot\|_\D\leq\|\cdot\|_{D(L(t))} \leq C_6 \|\cdot\|_\D$ for all\ $t\in[0,T]$, where $c_6:=((1+C_4)C_6)^{-1}.$
 \end{proof}
 
\begin{proposition}\label{1809a} Let Assumption \ref{ass}(i) hold, then for every $t\in[0,T]$ the operator $\big(L(t), D(L(t))\big)$ is the generator of a $C_0$-semigroup on $H$.
\end{proposition} 
\begin{proof} 
As a consequence of Lemma \ref{lm1} and Lemma \ref{lm2}, for every $t\in[0,T]$ the operator $L(t)=L_0+L_1(t)$ on $D(L(t))$ generates a $C_0$-semigroup on $H$, see \cite[Theo.~3.1.1]{Pazy}.
\end{proof}
\begin{proposition} \label{1809b} Let Assumption \ref{ass}(i) hold, then the family $(L(t))_{0\leq t\leq T}$ is stable on $H$.
\end{proposition}
\begin{proof}
By (\ref{19081}), $\|L_1(t)\|$ is uniformly bounded on $[0,T]$. Moreover, $L_0$ is a generator of $C_0$-semigroup of contractions, then by \cite[Theo.~5.2.3]{Pazy} the family $(L(t))_{0\leq t\leq T}$ is stable in $H$ with stability constants $1, C_4$.
\end{proof}
\begin{proposition}\label{1809c}
Let Assumption \ref{ass}(i) hold. Then $\D$ is $L(t)$-admissible for all $t\in[0,T]$ and the family $(\widehat{L}(t))_{0\leq t\leq T}$ of parts $\widehat{L}(t)$ of $L(t)$ in $\D$ is stable in $\D$.
\end{proposition}
\begin{proof}
 Since for every $t\in [0,T], L(t)$ is a generator of $C_0$-semigroup $(S_t(\tau))_{\tau\geq0}$ and $\D=D(L(t))$, we obtain that $\D$ is an invariant subspace of $(S_t(\tau))_{\tau\geq0}$ for all $t\in[0,T]$. Recall that the norms $\|\cdot\|_{D(L(t))}$ and $\|\cdot\|_{\D}$ are equivalent on $\D$ uniformly in $t\in[0,T]$, see Lemma \ref{lm3}. Hence, the restriction $(\widehat{S}_t(\tau))_{\tau\geq0}$ of $(S_t(\tau))_{\tau\geq0}$ to $\D$ is a $C_0$-semigroup on $(\D, \|\cdot\|_{\D})$, i.e., $\D$ is $L(t)$-admissible.

Consider the part $\widehat{L}(t)$ of $L(t)$ on $\D$, $t\in[0,T]$. By Lemma \ref{lm1}(i), $L_1(t)\in L(\D)$ for all $t\in[0,T]$. Hence, we have 
\begin{equation}\label{14111}
D(\widehat{L}(t))=\big\{w\in \D\big| L(t)w\in \D\big\}= D(L_0^2)\ \text{and}\  \widehat{L}(t)w=L(t)w\ \text{for all}\ w\in D(L_0^2).
\end{equation}
Since the family $((L(t),\D))_{0\leq t\leq T}$ is stable on $H$, the operators $\alpha Id-L(t): \D\subset H\longrightarrow H$ are surjective for all $\alpha>C_4$.
Combining with (\ref{14111}), $\alpha Id-\widehat{L}(t): D(L_0^2)\subset \D\longrightarrow \D$ are surjective for all $t\in[0,T]$ and $\alpha>C_4$. By Lemma \ref{lm1}(i) together with the skew-symmetry of $L_0$ we have $ \langle L(t)w,w\rangle_{\D}\leq C_5\langle w,w\rangle_{\D}$ for all $t\in [0,T]$ and $w\in D(L_0^2)$. Hence, $\langle (L(t)-C_5)w,w\rangle_\D\leq 0$ for all $w\in D(L_0^2)$. So, $\|(\alpha Id-L(t))w\|^2_\D\geq (\alpha-C_5)^2\|w\|_\D^2\ \text{for all}\ \alpha>C_5$ and  $w\in D(L_0^2).$ Let $m:=\max\{C_4, C_5\}$. Then $\|R(\alpha: L(t))\|\leq \frac{1}{\alpha-m}$ for all $\alpha>m$ and $t\in[0,T]$. Hence, $\big((\widehat{L}(t), D(L_0^2))\big)_{0\leq t\leq T}$ is stable on $\D$ with stability constants $1$ and $m$.
\end{proof}
         
\begin{remark}\label{rm1811} 
 (i) By Lemma \ref{lm1}(ii), Proposition \ref{1809a}, \ref{1809b}, and  \ref{1809c}, there exists an evolution system $(U(t,\tau))_{0\leq\tau\leq t\leq T}$ on $H$ corresponding to $(L(t), D(L(t)))_{0\leq t\leq T}$ in the sense of \cite[p.~129]{Pazy} satisfying:

(a) $\|U(t,\tau)\|\leq e^{(C_4(t-\tau))}$ \ for all $0\leq \tau\leq t\leq T$;

(b) for all $w\in \D$ and $\tau\in[0,T], \ \partial_t^+U(t,\tau)w\Bigl|_{t=\tau}=L(\tau)w$\ for a.e.~$t\in[\tau,T];$ 

(c) for all $w\in \D$ and $t\in(0,T],$ \ $\partial_\tau U(t,\tau)w = -U(t,\tau)L(\tau)w$ \ for a.e.~$\tau\in[t,T]$,\\
see \cite[Theo.~5.3.1]{Pazy}.

\noindent
(ii) Due to \cite[p.~136,137]{Pazy} there exists a bounded sequence $(U_n(t,\tau))_{n\in\N}$ in $L(\D)$ approximating $U(t,\tau)$ in the strong operator topology for all $0\leq \tau\leq t\leq T$. Moreover, $(\D,\|\cdot\|_\D)$ is reflexive. Hence, for all $0\leq\tau\leq t\leq T$ we have $U(t,\tau)(\D)\subset \D$ and $\|U(t,\tau)\|_{L(\D)}\leq e^{(m(t-\tau))}$, where $m$ is as in the proof of Proposition \ref{1809c}.

\noindent
(iii) As far as we know, the uniqueness of evolution system $(U(t,\tau))_{0\leq\tau\leq t\leq T}$ for the case $L(t)\in L^1([0,T],L(\D,H))$ as in \cite[Remark 5.3.2]{Pazy} is not clear. Due to \cite[Exam.~8.20(b)]{Rud}, the fundamental theorem of calculus does not hold for some continuous function, which are only a.e.~differentiable with integrable derivative. 
\end{remark}
Let $C([\tau,T], \D),\ 0\leq\tau\leq T,$ be the space of continuous functions on $[\tau,T]$ with values in $\D$ and $\alpha>0$. Define $$\|f\|_\alpha:=\sup_{t\in[\tau,T]}\big(\|f(t)\|_{\D}e^{-\alpha t}\big),\quad \  f\in C([\tau,T], \D).$$ Then $\big(C([\tau,T], \D), \|\cdot\|_\alpha\big)$ is a Banach space.
\begin{proposition}\label{pp1}
 Let Assumption \ref{ass}(i) hold and $(U(t,\tau))_{0\leq \tau\leq t\leq T}$ be the evolution system as in Remark \ref{rm1811}. Then for each $w\in \D$ and $\tau\in[0,T]$ there exists an unique $u_w^\tau\in C([0,T], \D)$ such that $u_w^\tau(t)=U(t,\tau)w$ for all $t\in[\tau, T]$. Moreover, $(U(t,\tau))_{0\leq\tau\leq t\leq T}$ is an almost strong evolution system on $H$ corresponding to $((L(t),\D))_{0\leq t\leq T}$ with initial value space $\D$ and satisfies \begin{equation}\label{26104b}
\partial_tU(t,\tau)w=L(t)U(t,\tau)w\quad\text{for all}\  w\in \D,\  t\in[\tau,T].
\end{equation}
\end{proposition}
\begin{proof} We have $U(t,\tau)(\D)\subset \D$, see Remark \ref{rm1811}(ii). First, we prove that $
\partial_t U(t,\tau)w=L(t)U(t,\tau)w$\ for all $w\in \D, t\in [\tau,T]$.

Let $(S(t))_{t\geq0}$ be the $C_0$ semigroup of contractions generated by $(L_0,\D)$, see Proposition \ref{1809a}. Following \cite[Theo.~1.2.4 and Theo.~4.1.3]{Pazy}, $S(t-\tau)(\D)\subset \D$ and $S(t-\tau)w$ is the unique solution of 
\begin{equation*}
\frac{du}{dt}(t)=L_0u(t),\ u(\tau)=w,\ w\in \D,\ \text{and}\ u\in C([\tau,T], \D).
\end{equation*}
We prove that the equation
\begin{equation}\label{08091}
\frac{du}{dt}(t)=L(t)u(t),\ u(\tau)=w,\ w\in \D,\ \text{and}\ u\in C([\tau, T], \D)
\end{equation}
has a unique solution $u_w^\tau$ and $u_w^\tau(t)=U(t,\tau)w$ for all $t\in[\tau, T]$, where $(U(t,\tau))_{0\leq\tau\leq t\leq T}$ is the evolution system as in Remark \ref{rm1811}. 

By Lemma \ref{lm1}(i) for each $w\in \D$ and $\tau\in[0,T]$ the following map is well-defined
\begin{equation*}\label{07092}
\begin{split}
C([\tau,T], \D) \ni u\longmapsto J_w^\tau u\in C([\tau,T],\D),  
\end{split}
\end{equation*}
where $ J^\tau_wu(t):=S(t-\tau)w+\int_\tau^tS(t-r)L_1(r)u(r)dr, \ t\in[\tau,T].$
 Since $(S(t))_{t\geq0}$ restricted to $\D$ is a contraction semigroup, together with Lemma \ref{lm1}(i) we have for arbitrary $u_1,u_2\in C([\tau,T], \D)$ and $t\in[\tau,T]$
\begin{multline*}
\|(J^\tau_wu_1-J^\tau_wu_2)(t)\|_{\D}e^{-\alpha t}\leq \int_\tau^te^{-\alpha t}\|S(t-r)L_1(r)\big(u_1(r)-u_2(r)\big)\|_{\D}dr\\
 \leq C_5\int_\tau^te^{-\alpha (t-r)}\|u_1(r)-u_2(r)\|_{\D}e^{-\alpha r}dr=C_5\|u_1-u_2\|_\alpha\int_\tau^te^{\alpha (r-t)}dr\leq \frac{C_5}{\alpha}\|u_1-u_2\|_\alpha.
\end{multline*}
We choose $\alpha>C_5$. Then by the Banach fixed point theorem, there exists a unique $u_w^\tau\in C([\tau,T], \D)$ such that $$u^\tau_w(t)=S(t-\tau)w+\int_\tau^tS(t-r)L_1(r)u^\tau_w(r)dr.$$ Moreover, using Lebesgue's dominated convergence and the closedness of $(L_0,\D)$, for all $t\in[\tau, T]$ we have
\begin{equation*}
\begin{split}
\frac{d}{dt}u^\tau_w(t)&=L_0S(t-\tau)w+\int_\tau^tL_0S(t-r)L_1(r)u^\tau_w(r)dr+L_1(t)u^\tau_w(t)\\
                      &=L_0\Big(S(t-\tau)w+\int_\tau^tS(t-r)L_1(r)u^\tau_w(r)dr\Big)+L_1(t)u^\tau_w(t)=L(t)u^\tau_w(t)
\end{split}
\end{equation*}
and $u^\tau_w(\tau)=w.$ Hence, $u^\tau_w(t)$ is a solution of (\ref{08091}) for $t\in[\tau,T]$. Similarly as in the proof of \cite[Theo.~5.4.2]{Pazy} we have 
\begin{equation}\label{26104a}
u^\tau_w(t)=U(t,\tau)w\quad\text{for all}\ t\in[\tau,T].
\end{equation}
That implies 
\begin{equation*}
\partial_tU(t,\tau)w=L(t)U(t,\tau)w\quad\text{for all}\  w\in \D,\  t\in[\tau,T].
\end{equation*}
 Due to measurability of $L(t)$, strong continuity of $U(t,\tau)$, Lemma \ref{lm1}(ii) together with Remark \ref{rm1811}(ii), by (\ref{26104b}) and \cite[Theo.~4.2.11]{Mikl} we have 
\begin{equation*}
\int_\tau^tL(r)U(r,\tau)wdr=U(t,\tau)w-w\quad\text{for all}\ w\in\D, t\in[\tau,T].\qedhere
\end{equation*}
\end{proof}   

We consider the family of linear operators $\big(L^*(t), D(L^*(t))\big)_{0\leq t\leq T}$ w.r.t. $\langle\cdot,\cdot\rangle_H$. Since for each $t\in[0,T]$, $\big(L(t), D(L(t))\big)$ generates a $C_0$-semigroup on the separable Hilbert space $H$, so does $\big(L^*(t), D(L^*(t))\big)$.
Since $L_1^*(t)\in L(H)$ for every $t\in[0,T]$, we have $L^*(t)=L^*_0+L_1^*(t)=-L_0+L_1^*(t)$. Note that on the subspace $\D$ we have\ $L_1^*(t)=\begin{pmatrix}
   0& \tilde{L}^*_1(t)\\
   0& 0
   \end{pmatrix}$, where
    \begin{equation}\label{06121}   \big(\tilde{L}^*_1(t)\v\big)(s)=\int_s^l\int_{s_1}^l\int_0^{s_2}\lambda(s_3,t)\partial_{s_3}\v(s_3)ds_3ds_2ds_1,\ \v\in H^{2,2}_{{\rm bc}}(0,l),\ s\in(0,l).
    \end{equation}
 Since $(L_0, \D)$ is skew-adjoint, we have the following chain of equalities of subspaces of $H$ $$\D^*:=\cap_{0\leq t\leq T}D(L^*(t))=D(L_0^*)=D(L_0)=\D,\ t\in[0,T].$$
 \begin{lemma}\label{30112}
Let Assumption \ref{ass}(i) hold, then $L_1^*(t)w\in \D$ for all $w\in \D$ and\ $t\in[0,T]$. Moreover, $D((L^*(t))^2)=D((L_0)^2)=H_{{\rm bc}}^{6,2}(0,l)\times H_{{\rm bc}}^{4,2}(0,l)$, independent of $t\in[0,T]$.
 \end{lemma}
 \begin{proof} The first statement can be obtained easily by using (\ref{06121}) together with Assumption \ref{ass}(i). The second statement is implied by the first one.
 \end{proof}
\begin{proposition}\label{pp10} Let Assumption \ref{ass}(i) hold and $A^*(t):=L^*(-t),\ t\in[-T,0]$. Then $(A^*(t))_{-T\leq t\leq 0}$ generates an almost strong evolution system $(V(t,\tau))_{-T\leq \tau\leq t\leq 0}$ with the initial value space $\D$. In particular, for all $t\in[\tau,0]$ we have 
\begin{equation}\label{20091}
\partial_t V(t,\tau)w=A^*(t)V(t,\tau)w,\ w\in \D.\qedhere  
\end{equation}      
\end{proposition}
\begin{proof}
Similarly as in the proof of Lemma \ref{lm1}(ii), Propositions \ref{1809a}, \ref{1809b}, and \ref{1809c} we can show that $\D$ is $A^*(t)$-admissible and $(A^*(t))_{-T\leq t\leq 0}$ is stable (in $H$) with some stability constants $M_1, m_1$. The family $(\widehat{A}^*(t))_{-T\leq t\leq0}$ of parts $\widehat{A}^*(t)$ of $A^*(t)$ in $\D$ is stable in $\D$. For all $t\in[0,T]$, $A^*(t)\in L\big(\D, H\big)$ and the map $[0,T]\ni t\longmapsto A^*(t)\in L\big(\D, H\big)$ is bounded and measurable. Hence, by \cite[Theo.~5.3.1]{Pazy}, there exists an evolution system $(V(t,\tau))_{-T\leq\tau\leq t\leq0}$ on $H$ as discussed in Remark \ref{rm1811}. Now the same technique as in the proof of  Proposition \ref{pp1} can be applied to conclude the proof.   
\end{proof}
\begin{remark}\label{rm22111}
Set $\langle\u,\v\rangle_{H^{6,2}_{{\rm bc}}(0,l)}:= b^3\int_0^l\langle\partial_{ssssss}\u,\partial_{ssssss}\v\rangle_{{\rm euk}}ds,\ \u, \v\in H^{6,2}_{{\rm bc}}(0,l).$ Then the space $(D(L_0^2),\langle\cdot,\cdot\rangle_{D(L_0^2)})$ is a separable Hilbert space, where $\langle\cdot,\cdot\rangle_{D(L_0^2)}$ is the inner product on the product space $H_{{\rm bc}}^{6,2}(0,l)\times H_{{\rm bc}}^{4,2}(0,l)$. Using similar ideas as in the proof of Lemma \ref{lm1}, one can prove that $L^*(t)\in L\big(D(L_0^2), \D\big)$ and the map $[0,T]\ni t\longmapsto \|L^*(t)\|_{L(D(L_0^2), \D)}\in\R$ is bounded.
\end{remark}
\begin{proposition}\label{rm22112}
Let $(U^*(t,\tau))_{0\leq\tau\leq t\leq T}$ be the family of Hilbert adjoints $U^*(t,\tau)$ of\ $U(t,\tau)$ w.r.t. $\langle\cdot,\cdot\rangle_H$. Then for all $u\in \D$ and\ $t\in(0,T]$, the map $[0, t]\ni\tau\longmapsto U^*(t,\tau)u\in H$ is differentiable and 
\begin{equation}\label{20092}
\partial_\tau U^*(t,\tau)u=-L^*(\tau)U^*(t,\tau)u,\ \tau\in[0,t].
\end{equation}
Moreover, for all $0\leq\tau\leq t\leq T$ we have\ $U^*(t,\tau)(D(L_0^2))\subset D(L_0^2)$  and there exist some constants $1\leq M_2<\infty,\ m_2\in\R$ such that
\begin{equation}\label{13101}
\|U^*(t,\tau)u\|_{D(L_0^2)}\leq M_2e^{m_2(t-\tau)}\|u\|_{D(L_0^2)}\quad\ \text{for all}\ u\in D(L_0^2).
\end{equation}
\end{proposition}
\begin{proof} 
Let $R(-\tau,-t):=U^*(t,\tau),\ 0\leq\tau\leq t\leq T$. We consider  $(R(t,\tau))_{-T\leq\tau\leq t\leq0}$ and prove that $V(t,\tau)=R(t,\tau), -T\leq\tau\leq t\leq0$. Let $u, w\in \D$, and $-T\leq \tau\leq r\leq t\leq0$. Since $R^*(t,r)=U^{**}(-r,-t)=U(-r,-t), -T\leq t\leq r\leq0$, we have $$\langle R(t,r)V(r,\tau)u,w\rangle=\langle V(r,\tau)u, R^*(t,r)w\rangle=\langle V(r,\tau)u,U(-r,-t)w\rangle.$$ 
Furthermore, since $V(r,\tau)u$ and $U(-r,-t)w$ are strongly differentiable in $H$ for $r\in[\tau,t]\subset [-T,0]$, see (\ref{26104b}) and (\ref{20091}), the function $[\tau,t]\ni r\longmapsto\langle V(r,\tau)u, U(-r,-t)w\rangle$ is differentiable on $[\tau,t]$ and 
\begin{multline*}
 \partial_r\langle V(r,\tau)u, U(-r,-t)w\rangle=\langle A^*(r)V(r,\tau)u, U(-r,-t)w\rangle + \langle V(r,\tau)u, -L(-r)U(-r,-t)w\rangle\\
=\langle A^*(r)V(r,\tau)u, U(-r,-t)w\rangle+\langle -A^*(r)V(r,\tau)u, U(-r,-t)w\rangle = 0.
\end{multline*}
That implies $\int_\tau^t\partial_r\langle R(t,r)V(r,\tau)u,w\rangle dr =0.$
Moreover, by \cite[Theo.~4.2.11]{Mikl}, we have $\int_\tau^t\partial_r\langle R(t,r)V(r,\tau)u,w\rangle dr= \langle V(t,\tau)u, w\rangle -\langle R(t,\tau)u,w\rangle$. Hence $\langle V(t,\tau)u, w\rangle =\langle R(t,\tau)u,w\rangle.$ Since $\D$ is dense in $H$ and $V(t,\tau), R(t,\tau)$ are linear bounded operators on $H$, we have 
 \begin{equation}\label{18111}
 V(t,\tau)=R(t,\tau) \quad \text{for all}\ -T\leq\tau\leq t\leq0.
 \end{equation}
  Hence, for all $u\in \D$\ and\ $-T<t\leq0$\ we have $ [-T,t]\ni\tau\longmapsto R(t,\tau)u\in H$ is differentiable. Back to positive time we  have
\begin{equation*}
\begin{split}
\partial_\tau U^*(t,\tau)u&=\partial_\tau R(-\tau,-t)u=\partial_\tau V(-\tau,-t)u\\                                           &=-A^*(-\tau)V(-\tau,-t)u=-A^*(-\tau)R(-\tau,-t)u=-L^*(\tau)U^*(t,\tau)u.
\end{split}
\end{equation*}
Similar as in Remark \ref{rm1811}(ii), combining the definition of $R(t,\tau)$ with (\ref{18111}) we obtain (\ref{13101}).
\end{proof}
 \begin{proposition}\label{0911p} For all $u\in D(L_0^2)$ and $t\in[0,T]$, the map $[0,t]\ni\tau\longmapsto \varrho(\tau):=U^*(t,\tau)u\in \D$ is an element in $C_w^1([0,T];\D)$ and $\varrho'_\D(\tau)=-L^*(\tau)\varrho(\tau)$ for all $\tau\in[0,t]$.
 \end{proposition}
 
 \begin{proof} First, applying Remark \ref{rm1811}(ii) to $(U^*(t,\tau))_{0\leq\tau\leq t\leq T}$ we obtain square integrability of the function $[0,t]\ni \tau \longmapsto U^*(t,\tau)u\in \D$ for all $w\in D(L_0^2)$. Second, by Proposition \ref{rm22112} we have $U^*(t,\tau)(D(L^2_0))\subset D(L_0^2)$ for all $0\leq\tau\leq t\leq T$. We shall prove that the map $[0,t]\ni \tau \longmapsto U^*(t,\tau)u\in \D$ is weakly differentiable for every $u\in D(L_0^2)$. Due to Lemma \ref{lm3}, for all $w\in D(L_0^2)$ we have
 \begin{equation}\label{13104}
  \langle U^*(t,\tau)u,w\rangle_{\D}=\langle L_0U^*(t,\tau)u,L_0w\rangle_H=\langle U^*(t,\tau)u,L_0^*L_0w\rangle_H.
  \end{equation}
By (\ref{20092}) together with (\ref{13104}) we have for all $\tau\in[0,t]$
\begin{equation}\label{dense}
 \begin{split}
\partial_\tau\langle U^*(t,\tau)u,w\rangle_\D&= \partial_\tau\langle U^*(t,\tau)u,L_0^*L_0w\rangle_H=\langle -L^*(\tau)U^*(t,\tau)u,L_0^*L_0w\rangle_H\\
&=\langle -L^*(\tau)U^*(t,\tau)u,w\rangle_{\D},\ w\in D(L^2_0).
 \end{split}
 \end{equation}
  Moreover, by Remark \ref{rm22111} and Proposition \ref{rm22112}, there exists a constant $0<C_7<\infty$ such that
\begin{equation}\label{sup}
\sup_{\tau\in[0,t]}\|-L^*(\tau)U^*(t,\tau)u\|_{\D}\leq C_7\|u\|_{D(L_0^2)}\quad\text{for all}\ u\in D(L_0^2).
\end{equation}
Due to (\ref{dense}) and (\ref{sup}) and the fundamental theorem of calculus, see \cite[Theo.~4.2.11]{Mikl}, for all $w\in D(L_0^2), \tau\in[0,t],$ and $h\ne0$ such that $\tau+h\in[0,t]$, we have 
\begin{multline}\label{dt222}
\Big\langle \frac{1}{h}(U^*(t,\tau+h)u-U^*(t,\tau)u),w\Big\rangle_{\D}=\frac{1}{h}\int_\tau^{\tau+h}\partial_r\langle U^*(t,r)u,w\rangle_{\D}dr\\
=\frac{1}{h}\int_\tau^{\tau+h}\langle -L^*(r)U^*(t,r)u,w\rangle_{\D}dr=-\Big\langle \frac{1}{h}\int_\tau^{\tau+h} L^*(r)U^*(t,r)u dr,w\Big\rangle_{\D}
\end{multline}
Since $D(L_0^2)$ is dense in $\D$, by (\ref{dt222}) for every $u\in D(L_0^2), \tau\in[0,t]$ and $h\ne0$ such that $\tau+h\in[0,t]$ we have
 \begin{equation}\label{19092}
\frac{1}{h}(U^*(t,\tau+h)u-U^*(t,\tau)u)=-\frac{1}{h}\int_\tau^{\tau+h} L^*(r)U^*(t,r)u dr.
\end{equation}
Hence, together with (\ref{sup}) we have
\begin{equation}\label{bdt}
\begin{split}
\left\| \frac{1}{h}\int_\tau^{\tau+h}( -L^*(r)U^*(t,r)u) dr\right\|_{\D}\leq \frac{1}{h}\int_\tau^{\tau+h}\| -L^*(r)U^*(t,r)u\|_{\D} dr \leq C_7\|u\|_{D(L_0^2)}.
\end{split}
\end{equation}
Combining (\ref{19092}) with (\ref{bdt}) we can conclude 
\begin{equation}\label{bcd}
\left\|\frac{1}{h}\big(U^*(t,\tau+h)u-U^*(t,\tau)u\big)\right\|_{\D}\leq C_7\|u\|_{D(L_0^2)}\ \ \text{for all}\ h\ne0\ \text{such that}\ \tau+h\in[0,t] .
\end{equation}
(\ref{dense}) and (\ref{bcd}) together now imply that $[0,t]\ni\tau\longmapsto\varrho(\tau)=U^*(t,\tau)u\in \D$ is weakly differentiable  for all $\tau\in[0,t]$ and
\begin{equation}\label{20093}
\varrho'_{\D}(\tau)=-L^*(\tau)U^*(t,\tau)u=-L^*(\tau)\varrho(\tau).
\end{equation}
By (\ref{sup}) and (\ref{20093}), the function $[0,T]\ni \tau\longmapsto \langle\varrho_{L^*}(\tau),w\rangle_\D\in\R$ is weakly differentiable and its derivative is square integrable for all $w\in\D$. Hence $[0,T]\ni \tau\longmapsto \langle\varrho(\tau),w\rangle_\D$ is in $H^{1,2}(0,T)$, i.e., for all $u\in D(L_0^2)$ and $t\in[0,T]$, the map $[0,t]\ni\tau\longmapsto U^*(t,\tau)u\in \D$ is an element in $C_w^1([0,T];\D)$. Together with (\ref{20093}), Proposition \ref{0911p} is proved.
\end{proof}

\begin{proof}[Proof of Theorem \ref{12101}]
Note that in our application $\D=\D^*=H^{4,2}_{bc}(0,l)\times H^{2,2}_{bc}(0,l)$. Combining Lemma \ref{lm1}(i), Lemma \ref{lm2}, and Lemma \ref{lm3} we can infer Assumption \ref{30111}(i). Remark \ref{rm1811}(i)(c) yields Assumption \ref{30111}(ii). Assumption \ref{30111}(iii) can be concluded from Proposition \ref{0911p}, where $Y^*:= D(L_0^2)=H_{{\rm bc}}^{6,2}(0,l)\times H_{{\rm bc}}^{4,2}(0,l)$ being characterized in Lemma \ref{30112}. The almost strong evolution system $(U(t,\tau))_{0\leq\tau\leq t\leq T}$ required in Theorem \ref{existence}(i) is constructed by Proposition \ref{pp1}. Its initial value subspace $Y$ is equal to $\D$. The condition in (\ref{07111}) can be concluded from 
\begin{equation*}\label{dktp}
{\rm Tr}\big((U(t,\tau)AQ^{\frac{1}{2}})(U(t,\tau)AQ^{\frac{1}{2}})^*\big)\leq \sigma^2e^{2C_4T}{\rm Tr}(Q) <\infty,
\end{equation*}
where we use Remark \ref{rm1811}(i)(a) and the assumption ${\rm Tr}(Q)<\infty$. Hence, by Theorem \ref{existence} and \ref{main} there exists a unique analytically weak solution to (\ref{ptae}),~(\ref{inicontt}),~(\ref{boundcontt}) for all initial values $(\boldsymbol{\xi}_1,\boldsymbol{\xi}_2)\in Y^*=H_{{\rm bc}}^{6,2}(0,l)\times H_{{\rm bc}}^{4,2}(0,l)$. Thus, Theorem \ref{12101} is proved.
 \end{proof}


\subsection{The non-homogeneous problem}\label{nonhomogen}

\begin{theorem} \label{12102} Let Assumption \ref{ass} hold. Then there exists a unique analytically weak solution to \eqref{pdae},
\eqref{inicond}, \eqref{boundcond}.
 \end{theorem} 
\begin{proof} Consider the function 
$$[0,l] \times [0,T] \mapsto {\bf v}(s,t) := (s-l){\bf e}_3 \in {\mathbb R}^3.$$
It is a strong solution to \eqref{pdae}, \eqref{inicond}, \eqref{boundcond} for $\f^{\rm det}=(g-\partial_s\lambda)\e_3$ and $\sigma = 0$.
Let $u$ be the unique analytically weak solution to \eqref{pdae}, \eqref{boundcontt} with $u(s,0) = \partial_tu(s,0) = \0$ for all $s \in [0,l]$
provided in Theorem \ref{12101}. Then ${\bf x} := {\bf u} + {\bf v}$ is the unique analytically weak solution to \eqref{pdae},
\eqref{inicond}, \eqref{boundcond} corresponding to a deterministic force $\f^{\rm det} - \partial_s\lambda\e_3$.
Since $\partial_s\lambda\e_3$ fulfills Assumption \ref{ass}(ii) and functions fulfilling this assumption form a linear vector space,
we do not obtain any restriction on the class of admissible deterministic forces.
\end{proof}

{\bf Acknowledgements:} The authors wish to thank Florian Conrad for helpful discussions. This work has been supported by Bundesministerium f\"ur Bildung und Forschung,
Schwerpunkt "Mathematik f\"ur Innovationen in Industrie und Dienstleistungen",
Verbundprojekt ProFil, 05M10UKB.
The financial support by DAAD through the PhD Program ``Mathematics in Industry and Commerce'' at TU Kaiserslautern is gratefully acknowledged.

\hfill
%
%
%

\end{document}